\documentclass{amsart}
\usepackage{amsmath}
\usepackage{amssymb}
\input xy
\xyoption{all}

\usepackage{color}
\newtheorem{theorem}{Theorem}[section]

\newtheorem{example}[theorem]{Example}
\newtheorem{Lemma}[theorem]{Lemma}

\newtheorem{proposition}[theorem]{Proposition}
\theoremstyle{remark}

\newcommand{\Cal}[1]{{\mathcal #1}}
\newcommand{\End}{\operatorname{End}}
\newcommand{\Aut}{\operatorname{Aut}}
\newcommand{\Tor}{\operatorname{Tor}}
\newcommand{\Hom}{\operatorname{Hom}}

\newcommand{\Ext}{\operatorname{Ext}}

\newcommand{\Mod}{\operatorname{Mod-\!}}

\newcommand{\diag}{{\rm diag}}
\newcommand{\smod}{\mbox{\rm mod-}}

\newcommand{\lMod}{\mbox{\rm -Mod}}

\newcommand{\op}{\operatorname{op}}

\DeclareMathOperator{\GL}{GL}
\DeclareMathOperator{\rad}{rad}
\DeclareMathOperator{\Morph}{Morph}

\newcommand{\add}{\mbox{\rm add}}

\newcommand{\soc}{\mbox{\rm soc}}

\newcommand{\ann}{\mbox{\rm ann}}

\newcommand{\cmat}{\left(\begin{array}}
\newcommand{\fmat}{\end{array}\right)}
\DeclareMathOperator{\Ob}{Ob}

\newcommand{\codim}{\mbox{\rm codim}}

\newcommand{\coker}{\mbox{\rm coker}}
\newcommand{\Ker}{\mbox{\rm Ker}}
\newcommand{\Coker}{\mbox{\rm Coker}}

\newcommand{\N}{\mathbb N}
\newcommand{\Z}{\mathbb{Z}}

\newcommand{\Q}{\mathbb{Q}}

\newcommand{\M}[1]{\Morph(\Mod #1)}
\newcommand{\E}[1]{E_{#1}}

\begin{document}

\title[MORPHISMS WITH SEMILOCAL ENDOMORPHISM RINGS]{HOMOMORPHISMS WITH SEMILOCAL ENDOMORPHISM RINGS BETWEEN MODULES}
\author{Federico Campanini}
\address{Dipartimento di Matematica, Universit\`a di Padova, 35121 Padova, Italy}
\email{federico.campanini@math.unipd.it}
 \author[Susan F. El-Deken]{Susan F. El-Deken}
\address{Department of Mathematics, Faculty of Science, Helwan University, Ain Helwan, 11790, Helwan, Cairo, Egypt}
 \email{Sfdeken@hotmail.com}
 \author[Alberto Facchini]{Alberto Facchini}
\address{Dipartimento di Matematica, Universit\`a di Padova, 35121 Padova, Italy}
 \email{facchini@math.unipd.it}
\thanks{The first and the third authors were partially supported by Dipartimento di Matematica ``Tullio Levi-Civita'' of Universit\`a di Padova (Project BIRD163492/16 ``Categorical homological methods in the study of algebraic structures'' and Research program DOR1828909 ``Anelli e categorie di moduli'').}
\begin{abstract} We study the category $\Morph(\Mod R)$ whose objects are all morphisms between two right $R$-modules. The behavior of objects of $\M R$ whose endomorphism ring in $\Morph(\Mod R)$ is semilocal is very similar to the behavior of modules with a semilocal endomorphism ring. For instance, direct-sum decompositions of a direct sum $\oplus_{i=1}^nM_i$, that is, block-diagonal decompositions, where each object  $M_i$ of $\Morph(\Mod R)$ denotes a morphism $\mu_{M_i}\colon M_{0,i}\to M_{1,i}$ and where all the modules $M_{j,i}$ have a local endomorphism ring $\End(M_{j,i})$, depend on two invariants. This behavior is very similar to that of direct-sum decompositions of serial modules of finite Goldie dimension, which also depend on two invariants (monogeny class and epigeny class). When all the modules $M_{j,i}$ are uniserial modules, the direct-sum decompositions (block-diagonal decompositions) of a direct-sum $\oplus_{i=1}^nM_i$ depend on four invariants. 
\end{abstract}


 \maketitle

\section{Introduction}

The study of block decompositions of matrices is one of the classical themes in Linear Algebra. We refer to the description of matrices up to the matrix equivalence $\sim$ defined, for any two rectangular $m\times n$ matrices $A$ and $B$, by $A\sim B$  if $B = Q^{-1} A P$ for some invertible $n\times n$ matrix $P$ and some invertible $m\times m$ matrix $Q$. Recently, the case of matrices with entries in an arbitrary local ring has sparked interest \cite{KV}. In \cite[Corollary~5.4]{AAF1}, B.~Amini, A.~Amini and A.~Facchini considered the case of diagonal matrices over local rings, proving that the equivalence of two such matrices depends on two invariants, called lower part and epigeny class. That is, if $a_1,\dots,a_n,b_1,\dots,b_n$ are elements of a local ring $R$, then $\diag(a_1,\dots,a_n)\sim\diag(b_1,\dots,b_n)$
if and only if
there are two permutations $\sigma,\tau$ of $\{1,2,\dots,n\}$ such that the cyclically presented right $R$-modules 
$R/a_iR$ and $R/b_{\sigma(i)}R$ have the same lower part and $R/a_iR$, $R/b_{\tau(i)}R$ have the same epigeny class, for every $i=1,2,\dots,
n$. Thus the block decomposition of a matrix with entries in a ring is not unique, that is, the blocks on two equivalent block-diagonal matrices are not uniquely determined.

The modern setting to study this kind of questions is considering the morphisms in the category $\Mod R$ of right modules over a ring $R$, which are the objects of a Grothendieck category $\Morph(\Mod R)$. More precisely, the objects of the
category $\Morph(\Mod R)$ are the $R$-module morphisms between right $R$-modules. We will
denote by $M$ the object $\mu_M\colon M_0\to M_1$. A morphism $u\colon M\to N$ in the category $\Morph(\Mod R)$ is a pair of $R$-module morphisms $(u_0,u_1)$ such that $u_1\mu_M=\mu_Nu_0$. Thus two objects $M,N$ of $\Morph(\Mod R)$ are isomorphic if and only if there exists a pair of $R$-module isomorphisms $u_0\colon M_0\to N_0$ and $u_1\colon M_1\to N_1$ such that $u_1\mu_M=\mu_Nu_0$. This is exactly the equivalence $\sim$ defined above by the formula $B = Q^{-1} A P$. For instance, in \cite{AAF1}, the third-named author considered the case of isomorphism of two objects $\oplus_{i=1}^nM_i\cong \oplus_{i=1}^nN_i$, where each $M_i$ is the left multiplication $\lambda_{a_i}\colon R_R\to R_R$ by $a_i \in R$ and each $N_j$ is the left multiplication $\lambda_{b_j}\colon R_R\to R_R$ by $b_j\in R$.

Now direct-sum decompositions of objects with a semilocal endomorphism ring follow particularly regular patterns. Thus, in this paper, we consider
the morphisms $\mu_M\colon M_0\to M_1$ whose endomorphism ring $\End_{\Morph(\Mod R)}(M)$ in the category $\Morph(\Mod R)$ is
semilocal. For instance, if $M_0, M_1$ are right $R$-modules with semilocal endomorphism rings in the category $\Mod R$, then all morphisms $\mu_M\colon M_0\to M_1$ have a semilocal endomorphism ring (Proposition~\ref{1}).

The content of the paper is as follows. Sections \ref{2} and~\ref{3} are devoted to the study of the basic properties of the category $\Morph(\Mod R)$. In particular, in Section~\ref{3} we consider some functors clearly related to morphisms, like domain, codomain, kernel and cokernel, and other functors linked with them (Propositions~\ref{3.2} and~\ref{3.5}). Direct-sum decompositions of an object $A$ of an additive category $\Cal A$ with splitting idempotents are described by a monoid $V(A)$ with order-unit, and when the endomorphism ring of $A$ is semilocal, the commutative monoid $V(A)$ turns out to be a Krull monoid. In Theorem~\ref{action}, we describe the relation between the monoids $V(M)$, $V(M_0)$ and $V(M_1)$. In Section~\ref{5}, we study morphisms whose endomorphism ring in $\Morph(\Mod R)$ is a ring of finite type, that is, is a ring that modulo its Jacobson radical is a direct product of finitely many division rings, and the morphisms whose endomorphism ring in $\Morph(\Mod R)$ is local (Theorem~\ref{5.3}). In Section~\ref{vil}, we consider morphisms between two modules $M_0$, $M_1$ with $\End(M_0)$ and $\End(M_1)$ local rings. The direct sums of these morphisms are described by two invariants, which we call domain class and codomain class (Theorem~\ref{6.1'}). We give an example of a direct sum of $n$ such morphisms with $n!$ pairwise non-isomorphic direct-sum decompositions. The case of morphisms between uniserial modules is treated in Section~\ref{7}. Endomorphism rings of uniserial modules have at most two maximal ideals, so that the endomorphism ring of a morphism between two uniserial modules has at most four maximal ideals. Thus finite direct-sums of morphisms between uniserial modules are described by four invariants (Theorem~\ref{7.3}).

In this paper all rings have an identity $1\neq 0$ and ring morphisms preserve $1$.
\section{The category $\Morph(\Mod R)$}\label{2}
		
Let $R$ be an associative ring with identity and $\Mod R$ the category of right $R$-modules.  Let $\Morph(\Mod R)$ denote the {\em morphism category}. The objects of this
category are the $R$-module morphisms between right $R$-modules. We will
denote by $M$ a generic object $\mu_M\colon M_0\to M_1$ of $\Morph(\Mod R)$. A morphism $u\colon M\to N$ in the category $\Morph(\Mod R)$ is a pair of $R$-module morphisms $(u_0,u_1)$ such that $u_1\mu_M=\mu_Nu_0$, that is, such that the diagram 
$$\xymatrix{
M_0 \ar[r]^{\mu_M} \ar[d]_{u_0} & M_1\ar[d]^{u_1} \\
N_0 \ar[r]^{\mu_N}  & N_1
 } $$ commutes.

We will denote by $\E M$ the endomorphism ring of the object $\mu_M\colon M_0\to M_1$ in the category $\M R$.

Let us examine the structure of the category $\Morph(\Mod R)$ more in detail. For every pair $M,N$ of objects of $\Morph(\Mod R)$, the group $$\Hom_{\Morph(\Mod R)}(M,N)$$ is a subgroup of the cartesian product $\Hom_{\Mod R}(M_0, N_0)\times \Hom_{\Mod R}(M_1, N_1)$.  
Thus, for every pair $M,N$ of objects of $\Morph(\Mod R)$,  addition is defined on each additive abelian group $\Hom_{\Morph(\Mod R)}(M,N),$ and we can set $u+u':=(u_0+u'_0, u_1+u'_1)$ for every $u=(u_0,u_1),u'=(u'_0,u'_1)\in \Hom_{\Morph(\Mod R)}(M,N)$.

The next theorem is well known \cite{Fossum, Green}. Since in those two references the result is stated for left $R$-modules and we need it for right ones, we briefly sketch some steps of the proof for later references.

\begin{theorem}\label{2.1} The category $\Morph(\Mod R)$ is equivalent to the category of right modules over the triangular matrix ring $T:=\left(\begin{smallmatrix}R&R\\ 0&R\end{smallmatrix}\right)$.\end{theorem}

\begin{proof}
The equivalence $F\colon \Morph(\Mod R)\to\Mod T$ is defined as follows. Given any object $M$ in $ \Morph(\Mod R)$, that is, a right $R$-module homomorphism $$\mu_M\colon M_0\to M_1,$$ consider the abelian group $M_0\oplus M_1$. The right $T$-module structure on $M_0\oplus M_1$ is given by the ring antihomomorphism $$\rho\colon T \to\End_\Z(M_0\oplus M_1)\cong \left(\begin{array}{cc}\End_\Z(M_0) & \Hom_\Z(M_1,M_0) \\ \Hom_\Z(M_0,M_1) & \End_\Z(M_1)\end{array}\right),$$ $$\rho\colon \left(\begin{array}{cc}r&s\\0&t\end{array}\right)\mapsto \left(\begin{array}{cc}\rho_r&0\\\rho_s\circ\mu_M&\rho_t\end{array}\right).$$ The functor $F$ assigns to each morphism $u=(u_0,u_1)\colon M\to N$ in $\Morph(\Mod R)$ the right $T$-module morphism $\left(\begin{smallmatrix}u_0&0\\0&u_1 \end{smallmatrix}\right)\colon M_0\oplus M_1\to N_0\oplus N_1$.

The quasi-inverse of $F$ is the functor $G\colon \Mod T \rightarrow \Morph(\Mod R)$ that associates to a right $T$-module $M_T$, that is, to a ring antihomomorphism $\rho\colon T\to\End_\Z(M)$, the morphism $\mu_M\colon Me_{11}\to Me_{22}$, as follows. Here $e_{ij}$ ($i,j=1,2$) is the matrix with $1$ in the $(i,j)$-entry and $0$ elsewhere. Right multiplication by $e_{11}$ is a group morphism $M\to M$, $m \mapsto me_{11}$, which is clearly an idempotent group morphism. Thus we have a direct-sum decomposition $M=Me_{11}\oplus Me_{22}$ of $M$ as an abelian group. Since there is a canonical ring homomorphism $R\to T$, $r\mapsto \left(\begin{smallmatrix}r&0\\0&r \end{smallmatrix}\right)$, every right $T$-module is a right $R$-module in a canonical way. From the equalities $\left(\begin{smallmatrix}r&0\\0&r \end{smallmatrix}\right)\left(\begin{smallmatrix}1&0\\0&0 \end{smallmatrix}\right)=\left(\begin{smallmatrix}1&0\\0&0 \end{smallmatrix}\right)\left(\begin{smallmatrix}r&0\\0&r \end{smallmatrix}\right)$, it follows that $Me_{11}$ is an $R$-submodule of $M_R$, and so is $Me_{22}$. Thus $M=Me_{11}\oplus Me_{22}$  is a direct-sum decomposition of $M_R$ as a right $R$-module. From the identity $e_{11}e_{12}e_{22}=e_{12}$, we get that $\rho(e_{22})\circ \rho(e_{12})\circ \rho(e_{11})= \rho(e_{12})$, so that $\rho(e_{12})(Me_{11})\subseteq Me_{22}$. Right multiplication $\rho(e_{12})\colon M\to M$ induces by restriction a group morphism $\rho(e_{12})|_{Me_{11}}^{Me_{22}}\colon Me_{11}\to Me_{22}$. From the equalities $\left(\begin{smallmatrix}r&0\\0&r \end{smallmatrix}\right)\left(\begin{smallmatrix}0&1\\0&0 \end{smallmatrix}\right)=\left(\begin{smallmatrix}0&1\\0&0 \end{smallmatrix}\right)\left(\begin{smallmatrix}r&0\\0&r \end{smallmatrix}\right)$, we have that $\rho(e_{12})\colon M\to M$ is a right $R$-module morphism. Thus $\mu_M:=\rho(e_{12})|_{Me_{11}}^{Me_{22}}\colon Me_{11}\to Me_{22},$ is an object of $\Morph(\Mod R)$, corresponding to the right $T$-module $M_T$. Every right $T$-module morphism $\alpha\colon M\to N$ is such that $\alpha(Me_{11})\subseteq Me_{11}$ and $\alpha(Me_{22})\subseteq Me_{22}$, and therefore, it is in matrix form of the type $\alpha=\left(\begin{smallmatrix}u_0&0\\0&u_1 \end{smallmatrix}\right)\colon Me_{11}\oplus Me_{22}\to Ne_{11}\oplus Ne_{22}$. The functor $G$ associates to $\alpha$ the morphism $(u_0,u_1)$ in $\Morph(\Mod R)$.
\end{proof}

By Theorem~\ref{2.1}, the category $\Morph(\Mod R)$ is a Grothendieck category.

\bigskip

Let $\{\,M_\lambda\mid \lambda\in \Lambda\,\}$ be a family of objects of $\Morph(\Mod R)$, where $\lambda$ ranges in an index set $\Lambda$. Thus $M_\lambda$ is an object $\mu_{M_\lambda}\colon M_{0,\lambda}\to M_{1,\lambda}$ for every $\lambda\in\Lambda$. The coproduct of the family $\{\,M_\lambda\mid \lambda\in \Lambda\,\}$ is the object $\bigoplus_{\lambda\in\Lambda} M_{\lambda}$, where $$\mu_{\bigoplus_{\lambda\in\Lambda} M_{\lambda}}\colon \bigoplus_{\lambda\in\Lambda} M_{0,\lambda}\to \bigoplus_{\lambda\in\Lambda} M_{1,\lambda}$$ is defined componentwise, with the canonical embeddings $e_{\lambda_0}\colon M_{\lambda_0}\to \bigoplus_{\lambda\in\Lambda} M_{\lambda} $ for every $\lambda_0\in\Lambda$.

The product of the family $\{\,M_\lambda\mid \lambda\in \Lambda\,\}$ is the object $\prod_{\lambda\in\Lambda} M_{\lambda}$, where $$\mu_{\prod_{\lambda\in\Lambda} M_{\lambda}}\colon \prod_{\lambda\in\Lambda} M_{0,\lambda}\to \prod_{\lambda\in\Lambda} M_{1,\lambda}$$ is defined componentwise, with the canonical projections $p_{\lambda_0}\colon \prod_{\lambda\in\Lambda} M_{\lambda} \to  M_{\lambda_0}$ for every $\lambda_0\in\Lambda$.

\medskip

Let us briefly consider the kernel and the cokernel of a morphism $$u=(u_0,u_1)\colon M\to N$$ in the category $\Morph(\Mod R)$. Clearly, the morphism $u$ induces a commutative diagram of right $R$-modules and right $R$-module morphisms
\[
\xymatrix{
0 \ar[r] &\ker(u_0)\ar[r]^{\varepsilon_0} \ar[d]_{\mu_M|}  & M_0 \ar[r]^{u_0}  \ar[d]_{\mu_M} & N_0 \ar[r]^{p_0}  \ar[d]^{\mu_N} & \coker(u_0)\ar[r]\ar[d]^{\overline{\mu_N}}  & 0 \\
0 \ar[r] &\ker(u_1) \ar[r]_{\varepsilon_1}& M_1 \ar[r]_{u_1}& N_1 \ar[r]_{p_1}& \coker(u_1) \ar[r]& 0.}
\]
The kernel of $u$ is the object $\mu_M|\colon\ker(u_0)\to\ker(u_1)$, where $\mu_M|$ denotes the restriction of $\mu_M\colon M_0\to M_1$ to the kernels, with the inclusion $\varepsilon=(\varepsilon_0,\varepsilon_1)$. The cokernel of $u$ is the object $\overline{\mu_N}\colon\coker(u_0)\to\coker(u_1)$, where $\overline{\mu_N}$ denotes the right $R$-module morphism induced by $\mu_N\colon N_0\to N_1$ on the cokernels, with the projection $p=(p_0,p_1)$.

\section{Some canonical functors}\label{3}

For any ring $R$, there are four canonical covariant additive functors  $$\Morph(\Mod R)\to \Mod R.$$ They are:

\smallskip

(1) The domain functor $D\colon \Morph(\Mod R)\to \Mod R$, which associates to each object $M$ of $\Morph(\Mod R)$ the right $R$-module $M_0$ and to any morphism $u=(u_0,u_1)$ in $\Morph(\Mod R)$ the right $R$-module morphism $u_0$ in $\Mod R$.

(2) The codomain functor $C\colon \Morph(\Mod R)\to \Mod R$, which associates to each object $M$ of $\Morph(\Mod R)$ the right $R$-module $M_1$ and to any morphism $u=(u_0,u_1)$ the right $R$-module morphism $u_1$.

(3) The kernel functor $\Ker\colon \Morph(\Mod R)\to \Mod R$, which associates to each object $M$ of $\Morph(\Mod R)$ the right $R$-module $\ker(\mu_M)$ and to any morphism $u=(u_0,u_1)\colon M\to N$ the restriction of the morphism $u_0\colon M_0\to N_0$, obtained by restricting the domain of $u_0$ to $\ker(\mu_M)$ and the codomain to $\ker(\mu_N)$.

(4) The cokernel functor $\Coker\colon \Morph(\Mod R)\to \Mod R$, which associates to each object $M$ of $\Morph(\Mod R)$ the right $R$-module $\coker(\mu_M)$ and to any morphism $u=(u_0,u_1)\colon M\to N$ the right $R$-module morphism induced by the morphism $u_1\colon M_1\to N_1$ on the cokernels $\coker(\mu_M)$ and $\coker(\mu_N)$.

\medskip

For any ring $R$, the canonical functor $$U\colon \Morph(\Mod R)\to \Mod R\times \Mod R,$$ which assigns to every object $M$ of $ \Morph(\Mod R)$ the object $(M_0,M_1)$ of $\Mod R\times \Mod R$ and to every morphism $u=(u_0,u_1)$ in $\Morph(\Mod R)$ the morphism $(u_0,u_1)$ in $\Mod R\times \Mod R$, is simply the product functor $D\times C$.

In terms of the categorical equivalence between the categories $\Morph(\Mod R)$ and $\Mod T$ (see Theorem~\ref{2.1} and its proof), we have that $D,C$ and $U$ assign to every right $T$-module $M_T$ the right $R$-modules $Me_{11}$, $Me_{22}$, and the pair of right $R$-modules
$(Me_{11},Me_{22})$, respectively. Thus $D$ can be identified with (that is, is naturally isomorphic to) the functor $-\otimes_TTe_{11}\colon\Mod T\to\Mod R$ and $C$ can be identified with the functor $-\otimes_TTe_{22}\colon\Mod T\to\Mod R$.

\medskip

We now use a terminology that can be found, for instance, in \cite[Section~2]{FP}. Recall that a ring morphism $\varphi\colon R\rightarrow S$ is  {\em local} if, for every $r\in R$, $\varphi(r)$ invertible in $S$ implies $r$ invertible in $R$ \cite{Camps 1993}.
If $\Cal A$ and $\Cal B$ are preadditive categories and $F\colon\Cal A\to\Cal B$ is an additive functor, the functor $F$ is {\em local} if, for every pair $A,A'$ of objects of $\Cal A$ and every morphism $f\colon A\to A'$ in $\Cal A$, $F(f)$ isomorphism in $\Cal B$ implies $f$ isomorphism in $\Cal A$, and $F$ is {\em isomorphism reflecting} if, for every pair $A,A'$ of objects of $\Cal A$, $F(A)\cong F(A')$ implies $A\cong A'$.
The functor $U\colon \Morph(\Mod R)\to \Mod R\times \Mod R$ is a faithful local functor that is not isomorphism reflecting. For instance, for every non-zero object $A_R$ of $\Mod R$, the identity $A_R\to A_R$ and the zero morphism $A_R\to A_R$ are two non-isomorphic objects of $\Morph(\Mod R)$ that become isomorphic objects of $\Mod R\times \Mod R$ when $U$ is applied. Notice that, via the faithful functor $U$, the category $\Morph(\Mod R)$ can be viewed as a subcategory of $\Mod R\times \Mod R$.

\medskip

Now let $I$ be the ideal of $T$ consisting of all the matrices $$\left(\begin{array}{cc}0&a\\ 0&0\end{array}\right)\in T,\qquad a\in R.$$

\begin{Lemma}\label{XXX}
The ideal $I$ is a two-sided ideal of $T$, nilpotent of index $2$, hence contained in the Jacobson radical $J(T)$ of $T$. Moreover, $T/I$ is isomorphic, as a ring, to the direct product $R\times R$, $_TI=Te_{12}\cong Te_{11}$ is a cyclic projective left $T$-module, and $I_R\cong R_R$ is a free right $R$-module.
\end{Lemma}

In the statement of Lemma~\ref{XXX}, we look at the right $T$-module $I_T$ as a right $R$-module $I_R$ via the canonical embedding $R\to T$, $r\mapsto \binom{r\ 0}{0\ r}$.

\begin{proof}
The left $T$-module isomorphism $_TI=Te_{12}\to Te_{11}$ associates to the matrix $\binom{0\ a}{0\ 0}\in I$ the matrix $\binom{a\ 0}{0\ 0}\in Te_{11}$. It is given by right multiplication by $e_{21}$.
\end{proof}

Let us compare the functors $D,C,\Ker,\Coker\colon \Morph(\Mod R)\to \Mod R$ defined above with the derived functors of the functor $-\otimes_TT/I\colon \Mod T\to\Mod T/I$. From Lemma~\ref{XXX}, we have that $\Tor_n^T(-, {}_TT/I)=0$ for every $n\ge2$. In the next proposition we compute $-\otimes_TT/I$ and $\Tor^T_1(-, {}_TT/I)$. By Theorem~\ref{2.1}, we will identify the two equivalent categories $\Morph(\Mod R)$ and $\Mod T$.

\begin{proposition}\label{3.2}
{\rm (a)} The functor $$-\otimes_TT/I\colon \Mod T\to\Mod T/I\cong\Mod R\times\Mod R$$ is naturally isomorphic to the functor $D\times\Coker\colon \Mod T\to\Mod R\times\Mod R$.

{\rm (b)} The functor $-\otimes_TI_R\colon \Mod T\to\Mod R$ is naturally isomorphic to the functor $D\colon \Mod T\to\Mod R$.

{\rm (c)} The functor $\Tor_1^T(-, {}_TT/I_R)\colon \Mod T\to\Mod R$ is naturally isomorphic to the functor $\Ker\colon \Mod T\to\Mod R$.
\end{proposition}

We omit the proof, which is a standard elementary calculation.

\medskip

Via Proposition~\ref{3.2}, the exact sequence \begin{equation}\label{bjip'}
\xymatrix@1{ 0\ar[r]& \Tor_1^T(M, T/I_R)\ar[r]& M\otimes_TI\ar[r] & M\ar[r] &M/MI\ar[r] & 0}
\end{equation}
becomes, for every object $M$ of $\Morph(\Mod R)$, the exact sequence
\begin{equation*}\label{bjip}
\xymatrix@1{ 0\ar[r]& \ker\mu_M\ar[r]& M_0\ar[r]^{\binom{0}{\mu_M}\ \ \ \ } & M_0\oplus M_1\ar[r]&M_0\oplus\coker\,\mu_M\ar[r] & 0.}
\end{equation*}

By Lemma \ref{XXX}, the left $T$-module $T/I$ has projective dimension $\le1$. We have a canonical short exact sequence \begin{equation}
\xymatrix@1{ 0\ar[r]& I\ar[r]& T\ar[r] & T/I \ar[r] & 0} \label{MMM}
\end{equation} of $T$-$R$-bimodules.

\begin{Lemma}\label{3.3} The short exact sequence {\rm  (\ref{MMM})} of left $T$-modules does not split. In particular, the left $T$-module $_TT/I$ has projective dimension $1$.\end{Lemma}

\begin{proof}
Assume the contrary, that is, that the short exact sequence of left $T$-modules (\ref{MMM}) splits. Then there is a left $T$-module morphism $g\colon {}_TT/I\to {}_TT$ that composed with the canonical projection $_TT\to{}_TT/I$ is the identity of $T/I$. Every left $T$-module morphism $g\colon {}_TT/I\to {}_TT$ is the right multiplication by an element $t\in T$ such that $It=0$. Thus (\ref{MMM}) splits if and only if there exists an element $t\in T$ with $It=0$ and $1-t\in I$. These two conditions imply that $I=I(1-t)\subseteq I^2=0$, a contradiction.
\end{proof}

From Lemma~\ref{3.3} we get that $\Ext_T^n(_TT/I,{}_TN)=0$ for every $n\ge 2$ and every left $T$-module $N$. We must now describe the category of left $T$-modules.

\begin{theorem}\cite[Section 1]{HV} The categories $T\lMod$ and $\Morph(R\lMod)$ are equivalent categories.\end{theorem}

\begin{proof} The category of left $T$-modules is isomorphic to the category of right $T^{\op}$-modules. Now $$T^{\op}=\left(\begin{array}{cc} R&R\\ 0&R\end{array}\right)^{\op}\cong \left(\begin{array}{cc} R^{\op}&0\\ R^{\op}&R^{\op}\end{array}\right).$$ By \cite[Section~1]{HV}, the category of right modules over the ring $\left(\begin{smallmatrix} R^{\op}&0\\ R^{\op}&R^{\op}\end{smallmatrix}\right)$ is equivalent to the category $\Morph(\Mod R^{\op})$, that is, to the category\linebreak $\Morph(R\lMod)$. \end{proof} 

Let us describe the categorical equivalence of the previous theorem in more detail.

The equivalence $F\colon \Morph(R\lMod)\to{}T\lMod$ is defined as follows. Given any object $N$ in $ \Morph(R\lMod)$, that is a left $R$-module morphism $\nu_N\colon N_0\to N_1$, we consider the abelian group $N_1\oplus N_0$. The left $T$-module structure on $N_1\oplus N_0$ is given by the ring homomorphism $$\lambda\colon T \to\End_\Z(N_1\oplus N_0)\cong \left(\begin{array}{cc}\End_\Z(N_1) & \Hom_\Z(N_0,N_1) \\  \Hom_\Z(N_1,N_0) & \End_\Z(N_0) \end{array}\right)$$ $$\lambda\colon \left(\begin{array}{cc}r&s\\0&t\end{array}\right)\mapsto \left(\begin{array}{cc}\lambda_r&\nu_N\circ\lambda_s \\ 0&\lambda_t\end{array}\right).$$ 

The quasi-inverse $G$ of $F$ is the functor $G\colon T\lMod  \rightarrow \Morph(R\lMod)$ that associates to a left $T$-module $_TN$, that is, to a ring morphism $\lambda\colon T\to\End_\Z(N)$, the morphism $\nu_N\colon e_{22}N\to e_{11}N$, as follows. Left multiplication by $e_{22}$ is an idempotent group morphism $N\to N$, $n \mapsto e_{22}n$. Hence there is a direct-sum decomposition $N=e_{22}N\oplus e_{11}N$ of $N$ as an abelian group. Notice that, via the canonical ring homomorphism $R\to T$, $r\mapsto \left(\begin{smallmatrix}r&0\\0&r \end{smallmatrix}\right)$, every left $T$-module is a left $R$-module in a natural way, so that $e_{22}N$ and $ e_{11}N$ are $R$-submodules of $_RN$, and $N=e_{22}N\oplus e_{11}N$  is a direct-sum decomposition of $_RN$ as a left $R$-module. From the identity $e_{11}e_{12}e_{22}=e_{12}$, we get that $\lambda(e_{11})\circ \lambda(e_{12})\circ \lambda(e_{22})= \lambda(e_{12})$, so that $\lambda(e_{12})(e_{22}N)\subseteq e_{11}N$. Left multiplication $\lambda(e_{12})\colon N\to N$ induces by restriction a left $R$-module morphism $\lambda(e_{12})|_{e_{22}N}^{e_{11}N}\colon e_{22}N\to e_{11}N$. Thus $\nu_N=\lambda(e_{12})|_{e_{22}N}^{e_{11}N}\colon e_{22}N\to e_{11}N$ is the object of $\Morph(R\lMod)$ corresponding to the left $T$-module $_TN$. 
 
 \bigskip
 
Now consider the exact sequence \begin{equation}\label{bjip''} 
\xymatrix{ 0\ar[r]& \Hom(_TT/I,{}_TN)\ar[r]&  \Hom(_TT,{}_TN)\ar[r]&\qquad\qquad\qquad\\ &\qquad\qquad\qquad\ar[r]&\Hom(_TI,{}_TN)\ar[r]&\Ext^1_T(_TT/I,{}_TN)\ar[r] & 0. }\end{equation} 

\begin{proposition}\label{3.5}
{\rm (a)} $\Hom(_TT/I,{}_TN)\cong\ann_NI=N_1\oplus\ker\nu_N$ for every left $T$-module $_TN$, so that the functor $\Hom(_TT/I,-)$ is naturally isomorphic to the product functor $C\times\Ker$.

{\rm (b)} The functor $\Hom_T(_TI_R,-)\colon T \lMod \to R\lMod$ is naturally isomorphic to the functor $C\colon T \lMod \to R\lMod$.

{\rm (c)} The functor $\Ext^1_T(_TT/I,-)\colon T \lMod \to R\lMod$ is naturally isomorphic to the functor $\Coker\colon T \lMod \to R\lMod$.
\end{proposition}
 
 We also omit the proof of this proposition, which is a standard elementary calculation.
 
\section{The functor $U$ and the monoid $V(M)$}\label{4}

One of the main aims of this paper is to study the morphisms $\mu_M\colon M_0\to M_1$ whose endomorphism ring $\E M$ is
semilocal. Recall that a ring $S$ is {\em semilocal} if $S/J(S)$ is a semisimple artinian ring, where $J(S)$ denotes the Jacobson radical of the ring $S$. A ring $S$ is semilocal if and only if the dual Goldie dimension $\codim(S_S)$ of the right regular module $S_S$ is finite, if and only if the dual Goldie dimension $\codim(_SS)$ of the left regular module $_S S$ is finite \cite[Proposition 2.43]{Facchini 1998}.
In this case, $\codim(S_S) = \codim(_S S)$ is equal to the Goldie dimension of the semisimple $S$-module $S/J (S)$. 

Among the several classes of modules with a semilocal endomorphism ring, we mention artinian modules, finitely presented modules over a semilocal ring, and finitely generated modules over a semilocal commutative ring. Other classes of modules with semilocal endomorphism rings can be found in \cite[6.2]{Facchini 2012}. The main properties of modules with a semilocal endomorphism ring are the cancellation property, the $n$-th root property, and the fact that the class of modules with a semilocal endomorphism ring is closed under direct summands and finite direct sums. Other properties of modules with a semilocal endomorphism ring can be found in \cite[6.1]{Facchini 2012}. All these properties carry over immediately to objects of $\Morph(\Mod R)$, that is, morphisms of right $R$-modules, with a semilocal endomorphism ring. For instance, every morphism with a semilocal endomorphism ring is the direct sum of a finite number of indecomposable morphisms. Here ``direct sum" means that the morphism has a block decomposition.

We have already said in the previous section that the functor  $$U\colon \Morph(\Mod R)\to \Mod R\times \Mod R,$$ which assigns to every object $M$ of $ \Morph(\Mod R)$ the object $(M_0,M_1)$ of $\Mod R\times \Mod R$, is faithful and local. An immediate corollary of this fact is:

\begin{Lemma} \label{Lem1} For every object $M$ of $\Morph(\Mod R)$, the canonical ring morphism $\varepsilon\colon\E M\to \End(M_0)\times \End(M_1)$, defined by $\varepsilon\colon (u_0,u_1) \mapsto (u_0,u_1)$, is a local morphism. \end{Lemma}

\begin{proof} A morphism $(u_0,u_1)$  in the morphism category $\Morph(\Mod R)$ is an isomorphism if and only if
both $u_0$ and $u_1$ are right $R$-module isomorphisms.\end{proof}

\begin{proposition}\label{1}
Let $M$ be an object of $\Morph(\Mod R)$ with 
$\End(M_0)$ and\linebreak $\End(M_1)$ semilocal rings. Then the endomorphism ring $\E M$ of the morphism $M$ in the category $\Morph(\Mod R)$ is semilocal.
\end{proposition}

\begin{proof}  By Lemma~\ref{Lem1}, the ring morphism
$$
\varepsilon\colon\E M\to \End(M_0)\times \End(M_1),\qquad\varepsilon\colon (u_0,u_1) \mapsto (u_0,u_1),
$$
is a local morphism. Since $\End(M_0)$ and $\End(M_1)$ are semilocal rings, their direct product $\End(M_0)\times \End(M_1)$ is semilocal \cite[(4) on page 7]{Facchini 1998}, so that $\E M$ is semilocal by \cite[Corollary~2]{Camps 1993}.
\end{proof}

Recall that, for any preadditive category $\Cal A$, the {\em Jacobson radical} $\Cal J_{\Cal A}$ of $\Cal A$ is the ideal of $\Cal A$ consisting, for every pair $(A,B)$ of objects of $\Cal A$, of all morphisms $f
\colon A\to B$ for which $1_A-gf$ has a left inverse for every morphism $g\colon B\to A$ in $\Cal A$. The kernel of any local functor $F\colon\Cal A\to \Cal B$ is contained in the Jacobson radical $\Cal J_{\Cal A}$ of $\Cal A$. 

For example, we will consider in Section~\ref{vil} the full subcategory $\Cal L$ of $\Mod R$ whose objects are all right $R$-modules with a local endomorphism ring. For any two objects $M,N$ of $\Cal L$, the Jacobson radical of $\Cal L$ is defined by $\Cal J_{\Cal L}(M,N)=\{\,f\in\Hom(M,N)\mid f$ is not an isomorphism$\,\}$. The ideal $\Cal J_{\Cal L}$ is a completely prime ideal of the category $\Cal L$ (we will recall  the definition of completely prime ideal in an additive category in Section~\ref{vil}).

\begin{proposition}\label{4.3}
In the embedding $U\colon \Morph(\Mod R)\to \Mod R\times \Mod R,$  if $u=(u_0,u_1)\colon M\to N$ is a morphism in the category $\Morph(\Mod R)$,\linebreak $u_0\in\Cal J_{\Mod R}(M_0,N_0)$ and $u_1\in\Cal J_{\Mod R}(M_1,N_1)$, then $$u=(u_0,u_1)\in\Cal J_{\Morph(\Mod R)}(M,N).$$
\end{proposition}

\begin{proof} Both functors $$U\colon \Morph(\Mod R)\to \Mod R\times \Mod R$$ and $$P\colon \Mod R\times \Mod R\to \Mod R/\Cal J_{\Mod R}\times \Mod R/\Cal J_{\Mod R}$$ are local functors, so that the composite functor $$PU\colon \Morph(\Mod R)\to \Mod R/\Cal J_{\Mod R}\times \Mod R/\Cal J_{\Mod R}$$ is a local functor. Kernels of local functors are contained in the Jacobson radical, and the kernel of the functor $PU$ consists exactly of the morphisms $u=(u_0,u_1)\colon M\to N$ in the category $\Morph(\Mod R)$ with $u_0\in\Cal J_{\Mod R}(M_0,N_0)$ and $u_1\in\Cal J_{\Mod R}(M_1,N_1)$.
\end{proof}

We will see in Example~\ref{5.2} that the implications in Lemma~\ref{Lem1} and Proposition~\ref{4.3} cannot be reversed.

\bigskip

Recall that an element $s$ of a commutative additive monoid $S$ is an {\em order-unit\/} if for every $x\in S$ there exist an integer $n\ge 0$ and an element $y\in S$ such that $x+y=ns$.
We say that {\em idempotents split} in a category $\Cal A$, or that $\Cal A$ {\em has splitting idempotents}, if every idempotent endomorphism in $\Cal A$ has a kernel. For an object $A$ of an additive category $\Cal A$ with splitting idempotents, let $\add(A)$ denote the class of all objects of $\Cal A$ isomorphic to direct summands of $A^n$ for some integer $n\ge 0$. Define an equivalence relation $\sim$ on $\add(A)$ setting, for every $C,C'\in\add(A)$, $C\sim C'$ if $C$ and $C'$ are isomorphic objects of $\Cal A$. Let $\langle C\rangle$ denote the equivalence class modulo $\sim$ of an object $C$ of $\add(A)$ and $V(A):=\add(A)/{}\!\!\sim{}=\{\, \langle C\rangle\mid C\in\add(A)\,\}$ the quotient class modulo $\sim$. Consider the operation $+$ on $V(A)$ defined by $\langle C \rangle + \langle C' \rangle := \langle C\oplus C' \rangle$ for every $C,C' \in \add(A)$. Then the quotient class $V(A)$ turns out to be a (possibly large) commutative monoid with respect to the operation~$+$, and $\langle A\rangle$ is an order-unit in $V(A)$. 

More generally, every category $\Cal A$ has a skeleton $V(\Cal A)$, that is, a full, isomorph\-ism-dense subcategory in which no two distinct objects are isomorphic. It is well known that any two skeletons of $\Cal A$ are isomorphic and are equivalent to $\Cal A$.

The functor $U$ induces a monoid morphism on the monoid $V(M)$ of isomorphism classes of direct summands of finite direct sums of copies of an object $M$ of $\Morph(\Mod R)$. It is the monoid morphism $\Psi:V(M)\to V(M_0)\times V(M_1)$ defined by $\langle C\rangle\mapsto (\langle C_0\rangle,\langle C_1\rangle)$ for every object $C$, that is, $\mu_C\colon C_0\to C_1$, in $\add(M)$.

\begin{theorem}\label{action}
The monoid morphism $\Psi:V(M)\to V(M_0)\times V(M_1)$ is a morphism of monoids with order-unit, is onto, and the inverse image via $\Psi$ of any element $(\langle C_0\rangle,\langle C_1\rangle)$ of the codomain $V(M_0)\times V(M_1)$ is the set of all orbits with respect to the action of the group $\Aut(C_1)\times\Aut(C_0)$ on the set $\Hom_R(C_0,C_1)$.
\end{theorem}

\begin{proof} Let $(\langle C_0\rangle,\langle C_1\rangle)$ be an element in the codomain $V(M_0)\times V(M_1)$. Its inverse image via $\Psi$ consists of all morphisms $f\colon C_0\to C_1$ modulo the equivalence relation $\sim$ induced by isomorphism in $\Morph(\Mod R)$. That is, the inverse image of $(\langle C_0\rangle,\langle C_1\rangle)$ is $\Hom_R(C_0,C_1)/\!\!\sim{}={}\{\,[g]_{\sim}\mid g\colon C_0\to C_1\,\}$, where $[g]_\sim$ indicates the equivalence class of any $g$ modulo $\sim$. Now if $g,g'\colon C_0\to C_1$, then $g\sim g'$ if and only if there exists an isomorphism $u=(u_0,u_1)$ in $\Morph(\Mod R)$, that is, if and only if there exist an automorphism $u_0$ of $C_0$ and an automorphism $u_1$ of $C_1$ with $u_1g=g'u_0$. Thus the direct product $G:=\Aut(C_1)\times\Aut(C_0)$ of the two automorphism groups $\Aut(C_i)$ of the right $R$-modules $C_i$ acts on the set $\Hom_R(C_0,C_1)$ via the action defined, for every $(u_1,u_0)\in \Aut(C_1)\times\Aut(C_0)$ and every $g\in \Hom_R(C_0,C_1)$, by $(u_1,u_0)g:=u_1gu_0^{-1}$. Clearly, two elements $g,g'$ of $\Hom_R(C_0,C_1)$ are in the same orbit if and only if $g\sim g'$.\end{proof}

We have already remarked that $U$ is not isomorphism-reflecting. Equivalently, the monoid morphism $\Psi$ is not injective.

When both right $R$-modules $M_0$ and $M_1$ have a semilocal endomorphism ring, then the three monoids $V(M_0)$, $V(M_1)$ and $V(M)$ are Krull monoids (Proposition~\ref{1} and \cite[Theorem~3.4]{Facchini 2002}). 

\begin{example}{\rm Artinian modules have semilocal endomorphism rings \cite{Camps 1993}. In \cite{FHLV} it was shown that for every integer $n\ge 2$, there exists an artinian module $A_R$ over a suitable ring $R$ which is a direct sum of $t$ indecomposable submodules for every $t=2,3,\dots,n$. Consider the identity morphism $1_A\colon A_R\to A_R$. Then, in the category $\Morph(\Mod R)$, the object $1_A$ is the direct sum of $t$ indecomposable objects of $\Morph(\Mod R)$ for every $t=2,3,\dots,n$.}\end{example}

\begin{example}{\rm 
Let $k$ be a field and $W_0,W_1$ be two non-zero finite dimensional vector spaces over $k$, of dimension $n$ and $m$ respectively. The action of the group $\Aut(W_1)\times\Aut(W_0)=\GL(W_1)\times\GL(W_0)$ on the set $\Hom_k(W_0,W_1)$ considered in Theorem~\ref{action}, is such that two $m\times n$ matrices $A,B\in \Hom_k(W_0,W_1)$ are in the same orbit, that is, are equivalent modulo the relation $\sim$, if and only if they are equivalent, that is, there exist an invertible $n\times n$ matrix $P$ and an invertible $m\times m$ matrix $Q$ such that $B = Q^{-1} A P$. It is well know that two $m\times n$ matrices are equivalent if and only if they have the same rank, and that a canonical representative for the equivalent matrices of a fixed rank $r$ is given by the $m\times n$ matrix $$
\begin{pmatrix}
1 & 0 & 0 & & \cdots & & 0 \\
0 & 1 & 0 & & \cdots & & 0 \\
0 & 0 & \ddots & & & & 0\\
\vdots & & & 1 & & & \vdots \\
 & & & & 0 & & \\
 & & & & & \ddots &  \\
0 & & & \cdots & & & 0
\end{pmatrix},
$$
where the number of $1$'s on the diagonal is equal to $r$. Thus every morphism $\mu_W\colon W_0\to W_1$ is the direct sum in the category $\Morph(\Mod k)$ of the three indecomposable objets $k\to 0$, $0\to k$ and $1\colon k\to k$. This direct-sum decomposition in $\Morph(\Mod k)$ is unique up to isomorphism because the endomorphism rings of the three objects $k\to 0$, $0\to k$ and $1\colon k\to k$ are all isomorphic to $k$ and, therefore, they are three objects with a local endomorphism ring. Notice that any object $\mu_W\colon W_0\to W_1$ in $\Morph(\Mod k)$ has a semilocal endomorphism ring of dual Goldie dimension $\le n+m$, so that all the monoids in the rest of this example will be Krull monoids. Let $\N_0$ indicate the additive monoid of non-negative integers. The monoid $V(\Morph(\smod k))$ is isomorphic to the additive monoid $\N_0^3$, and the monoid morphism induced by the local functor $U=D\times C\colon \Morph(\smod k)\to\smod k\times\smod k$ is the morphism $\N_0^3\to\N_0\times\N_0$, $(a,b,c)\to (a+c,b+c)$.

The three objects $k\to 0$, $0\to k$ and $1\colon k\to k$ of $\Morph(\smod k)$ correspond to the three right $T$-modules $e_{11}T/e_{11}J(T)$, $e_{22}T\cong e_{11}J(T)$ and $e_{11}T$, respectively. Notice that all these three $T$-modules are uniserial (the first two $T$-modules are simple). Thus every finitely generated right $T$-module is a direct sum of uniserial modules.

For any fixed object $M$ of $\Morph(\smod k)$, $\mu_M\colon M_0\to M_1$, the objects in the category $\add(M)$ are always direct sums of the three objects $k\to 0$, $0\to k$ and $1\colon k\to k$, but there are no copies of $k\to 0$ if the morphism $\mu_M$ is injective, no copies of $0\to k$ if the morphism $\mu_M$ is surjective, and no copies of $k\to k$ if the morphism $\mu_M$ is the zero morphism. Therefore, in order to describe the morphism $\Psi:V(M)\to V(M_0)\times V(M_1)$ of monoids with order-unit (see Theorem~\ref{action}), we must distinguish several cases, according to weather $\mu_M$ is injective or not, surjective, the zero morphism, or the finite dimensional vector spaces $M_0$ and $M_1$ are zero or not. For instance, for $\mu_M\colon M_0\to M_1$ injective but not surjective and $M_0\ne 0$, we have that every object in the category $\add(M)$ is a direct sum of the two objects $0\to k$ and $1\colon k\to k$, so that the morphism $V(M)\to V(M_0)\times V(M_1)$ of monoids with order-unit is the morphism $\N_0^2\to \N_0\times\N_0$, $(b,c)\mapsto (c,b+c)$. Suppose that $M_0,M_1$ have dimension $n$ and $m$, respectively. Then the injective, but not surjective, mapping $\mu_M\colon M_0\to M_1$ has rank $n$, and is the direct sum of $n$ copies of $k\to k$ plus $m-n\ge1$ copies of $0\to k$. Thus the monoid morphism $V(M)\cong\N_0^2\to V(M_0)\times V(M_1)\cong \N_0\times\N_0$, induced by the functor $U=C\times D\colon \Morph(\smod k)\to\smod k\times\smod k$, maps the order-unit $(m-n,n)$ of the monoid $V(M)\cong\N_0^2$ to $(n,m)\in\N_0\times\N_0$, and maps the arbitrary element  $(b,c)$ of $V(M)\cong\N_0^2$ to the element $(c,b+c)$ of $V(M_1)\cong \N_0\times\N_0$.}\end{example}

\section{Rings of finite type}\label{5}

Recall that a ring $S$ is said to be {\em of type $n$} if $S/J(S)$ is a direct product of $n$ division rings or, equivalently, if $S$ has exactly $n$ maximal right ideals,  which are all two-sided ideals of $S$ \cite{FPFT}. The ring $S$ is a ring {\em of finite type} if it has type $n$ for some integer $n\ge 1$. If a ring $S$ has finite type, then the type $n$ of
$S$ coincides with the dual Goldie dimension of $S_S$ \cite[Proposition 2.43]{Facchini 1998}.
A ring $S$ has type 1 if and only if it is a local ring, if and only if there is a
local morphism of $S$ into a division ring. More generally, rings of finite type are the rings with a local morphism into the direct product of finitely many division rings \cite[Proposition~2.1]{FPFT}. A completely prime ideal $P$ of a ring $S$ is a proper ideal $P$ of $S$ such that, for every $x,y\in S$, $xy\in P$ implies that either $x \in P$ or $y \in P$. 

\begin{proposition}\label{finite type}
Let $M$ be an object of $\Morph(\Mod R)$. If $\End_R(M_0)$ and $\End_R(M_1)$ are rings of type $m$ and $n$, respectively, then $\E M$ has type $\le m+n$. Moreover, if $I_1,\dots,I_n$ are the $n$ maximal ideals of $\End_R(M_0)$ and $K_1,\dots,K_m$ are the $m$ maximal ideals of $\End_R(M_1)$, then the at most $n+m$ maximal ideals of $\E M$ are among the completely prime ideals $(I_t\times \End_R(M_1))\cap \E M$ (where $t=1,\dots,n$) and $(\End_R(M_0)\times K_q)\cap \E M$ (where $q=1,\dots,m$).
\end{proposition}

\begin{proof} Let $I_t$ $(t=1,\dots,n)$ be the $n$ maximal ideals of the ring $\End_R(M_0)$ of type~$n$. Then the canonical projection
$$
\End_R(M_0)\to \End_R(M_0)/J( \End_R(M_0))\cong \prod_{t=1}^n \End_R(M_0)/I_t
$$
is a local morphism. Similarly for the canonical projection
$$
\End_R(M_1)\to \End_R(M_1)/J( \End_R(M_1))\cong \prod_{q=1}^m \End_R(M_0)/K_q.
$$
Therefore, there is a canonical local morphism
$$
\E M\to \prod_{t=1}^n \End_R(M_0)/I_t \times \prod_{q=1}^m \End_R(M_0)/K_q
$$
onto the direct product of $n+m$ division rings. By \cite[Proposition~2.1]{FPFT}, the ring $\E M$ is a ring of type $\le n+m$. Looking at the proof of that result, one sees that the maximal ideals of $\E M$ are among the kernels of the $n+m$ canonical projections, which concludes the proof.
\end{proof}

\begin{example}\label{5.2}{\rm 
We have already seen that the inclusion $$\varepsilon\colon\E M\to \End(M_0)\times \End(M_1)$$ is a local morphism. If we identify $\E M$ with its image in $\End(M_0)\times \End(M_1)$, then we have that
$$
(J(\End(M_0))\times J(\End(M_1)))\cap \E M \subseteq J(\E M).
$$
Moreover, if both $\End(M_0)$ and $\End(M_1)$ are rings of finite type, then so is $\E M$. The following example shows that $(1)$ the previous inclusion involving the Jacobson radicals can be proper and $(2)$ it can occur that $E_M$ is a ring of finite type but neither $\End(M_0)$ nor $\End(M_1)$ are.
Let $k$ be any field. Consider the object $\mu_M\colon k^2\rightarrow k^2$ of $\Morph(\Mod k)$ given by $(x,y)\mapsto (x,0)$. Then $\mu_M$ is represented by the $2\times 2$ matrix $$M=\left(\begin{array}{cc}1 & 0 \\ 0 & 0
\end{array}\right).$$
The endomorphism ring $E_M$ of $M$ is given by the set of all pairs of matrices $(A_0,A_1) \in M_2(k)\times M_2(k)$ such that $MA_0=A_1M$. An easy computation shows that $E_M$ consists exactly of all the pairs $(A_0,A_1) \in M_2(k)\times M_2(k)$ of the form
$$
(A_0,A_1)=\left(\left(\begin{array}{cc} u & 0 \\ v & w \end{array}\right),
\left(\begin{array}{cc}
u & x \\
0 & y\\
\end{array}\right)
\right)\quad \mbox{for some } u,v,w,x,y \in k.$$
In particular, $E_M$ is a subring of $\left(\begin{smallmatrix}k & 0 \\ k & k\end{smallmatrix}\right)\times \left(\begin{smallmatrix}k & k \\ 0 & k\end{smallmatrix}\right)$. The nilpotent ideal $\left(\begin{smallmatrix}0 & 0 \\ k & 0\end{smallmatrix}\right)\times \left(\begin{smallmatrix}0 & k \\ 0 & 0\end{smallmatrix}\right)$ of $E_M$ is contained in the Jacobson radical of $E_M$. It follows that $0=E_M\cap (J(M_2(k))\times J(M_2(k)))\subset J(E_M)$.
Moreover, it is easy to see that the ring $E_M$ is a ring of type~3. Its maximal right ideals are the completely prime two-sided ideals
$$
I_1:=\left\lbrace \left(\left(\begin{array}{cc} 0 & 0 \\ v & w \end{array}\right),
\left(\begin{array}{cc}
0 & x \\
0 & y\\
\end{array}\right)
\right)\in E_M\mid v,w,x,y \in k\right\rbrace ,
$$
$$
I_2:=\left\lbrace \left(\left(\begin{array}{cc} u & 0 \\ v & 0 \end{array}\right),
\left(\begin{array}{cc}
u & x \\
0 & y\\
\end{array}\right)
\right)\in E_M\mid u,v,x,y \in k\right\rbrace ,
$$
$$
I_3:=\left\lbrace \left(\left(\begin{array}{cc} u & 0 \\ v & w \end{array}\right),
\left(\begin{array}{cc}
u & x \\
0 & 0\\
\end{array}\right)
\right)\in E_M\mid u,v,w,x \in k\right\rbrace .
$$}
\end{example}

We conclude this section characterizing morphisms with local endomorphism rings.

\begin{theorem}\label{5.3}
Let $M$ be any object of $\Morph(\Mod R)$, $\E M$ its endomorphism ring in $\Morph(\Mod R)$, $\varepsilon\colon \E M\to\End(M_0)\times\End(M_1)$ the inclusion, $\pi_i\colon \End(M_0)\times\End(M_1)\to \End(M_i)$, for $i=0,1$, be the canonical projections and $E_i:=\pi_i\varepsilon(\E M)$. Then the endomorphism ring $E_M$ of the object $M$ is local if and only if one of the following three conditions holds:

{\rm (1)} $M_0=0$ and $\End(M_1)$ is a local ring.

{\rm (2)} $M_1=0$ and $\End(M_0)$ is a local ring.

{\rm (3)} $M_0\ne0$, $M_1\ne0$ and, for every endomorphism $u=(u_0,u_1)\in \E M$:

\noindent {\rm (a)} either $u_0$ or $1-u_0$ is invertible in $E_0$, and

\noindent {\rm (b)} $u_0$ is invertible in $E_0$ if and only if $u_1$  is invertible in $E_1$.
\end{theorem}

\begin{proof}
Suppose that the endomorphism ring $\E M$ in $\Morph(\Mod R)$ is local. If $M_0=0$, then $\mu_M=0$, and so $\End(M_1)\cong \E M$ is local. Similarly if $M_1=0$. Suppose $M_0\ne 0$ and $M_1\ne 0$.
Notice that $M_0\ne 0$ and $M_1\ne 0$ imply that $1\ne 0$ in both rings $\End(M_0)$ and $\End(M_1)$, hence in both of their subrings $E_0$ and $E_1$. Hence $E_0$ and $E_1$ are non-trival homomorphic images of the local ring $\E M$. If $u=(u_0,u_1)\in \E M$, and $u_0$ is not invertible in $E_0$, then $1-u_0$ is invertible in  $E_0$, because $E_0$ is local. This proves that condition (a) in (3) holds. Moreover, the rings $E_i$ are homomorphic images of the local ring $E_M$, so that the kernel of the surjective morphism $E_M\to E_i$ is contained in the Jacobson radical (which is the maximal ideal) of $E_M$. Hence the image of the maximal ideal of $E_M$ (which is the set of non-invertible elements of $E_M$) is mapped exactly onto the maximal ideal of $E_i$. It follows that $u=(u_0,u_1)$ is an automorphism of $M$ if and only if $u_i$ is invertible $E_i$. Thus $u_0$ is invertible in $E_0$ if and only if $u$ is an automorphism of $M$, if and only if $u_1$ is invertible in $E_1$.

For the converse, it is clear that (1) and (2) imply $\E M$ local. If (3) holds, for every endomorphism $u=(u_0,u_1)\in \E M$ that is not an automorphism, we have that either $u_0$ is not an automorphism of $M_0$ or $u_1$ is not an automorphism of $M_1$. Hence $u_0$ is not invertible in $E_0$ or $u_1$  is not invertible in $E_1$. By (b), the elements $u_0$ and $u_1$ are not invertible in $E_0$ and $E_1$, respectively.  Now $E_0$ is a local ring by (a). Similarly, $E_1$ is a local ring by (a) and (b). It follows that $1-u_0$ and $1-u_1$ are invertible in $E_1$ and $E_2$, respectively.  Thus $1-u$ is invertible in $E_M$, i.e., the ring $E_M$ is local.
\end{proof}

As far as the statement and the proof of Theorem \ref{5.3} are concerned, notice that the ring $\E M$ is a subdirect product of the two rings $E_0$ and $E_1$. Moreover, the embedding $\E M\hookrightarrow E_0\times E_1$ is a local morphism.

\begin{Lemma}
$\E M$ is semilocal if and only if  two rings $E_0$ and $E_1$ are semilocal.
\end{Lemma}

\begin{proof}
$(\Rightarrow)$ Because both the rings $E_i$ are homomorphic images of $\E M$. $(\Leftarrow)$ Because the morphism $\E M\to E_0\times E_1$ is local.). Notice that $\E M$ always has the two two-sided ideals $\ker(\pi_i\varepsilon)$, whose intersection is the zero ideal. By Theorem \ref{5.3}, the ring $\E M$ is local if and only if both the rings $E_0$ and $E_1$ are local and $(\pi_0\varepsilon)^{-1}(J(E_0))=(\pi_1\varepsilon)^{-1}(J(E_1))$.
\end{proof}
We are exactly in the setting of \cite[Abstract]{FP}. We have the Grothendieck category $\Morph(\Mod R)$, the pair of ideals $\ker(D)$ and $\ker(C)$ in the category $\Morph(\Mod R)$ (they are the kernels of the functors $D,C\colon \Morph(\Mod R)\to\Mod R$ defined in the first paragraph of Section~\ref{3}), and we have the canonical functor $$P\colon\Morph(\Mod R)\to \Morph(\Mod R)/\ker(D)\times \Morph(\Mod R)/\ker(C),$$ which is a local functor.
In the terminology of \cite[Section~4]{Rogelio}, the category $$\Morph(\Mod R)$$ is a subdirect product of the two factor categories $\Morph(\Mod R)/\ker(D)$ and $\Morph(\Mod R)/\ker(C)$.

\begin{proposition}
Let $M$ be an object of $\Morph(\Mod R)$ and assume that $\End_R(M_0)$ and $\End_R(M_1)$ are rings of finite type. Then $M$ has a local endomorphism ring if and only if there exists $i=0,1$ such that for every endomorphism $u=(u_0,u_1)\in \E M$ both the following conditions hold:

\noindent {\rm (a)} either $u_i$ or $1-u_i$ is an automorphism of $M_i$, and

\noindent {\rm (b)} if $u_i$ is an automorphism of $M_i$, then $u$ is an automorphism of $M$.

\end{proposition}

\begin{proof}
Assume that $\E M$ is local. For every $u=(u_0,u_1)\in \E M$, either $u$ or $1-u$ is invertible, so either $u_i$ or $1-u_i$ is an automorphism of $M_i$ for every $i=0,1$.

Now, let $n$ and $m$ be the types of $\End(M_0)$ and $\End(M_1)$, respectively. As a trivial case, we have that if $n=0$ (that is, if $M_0=0$), then $\E M \cong\End(M_1)$ is a local ring and (b) follows. Similarly for $m=0$. Thus we can assume $n,m\geq 1$. Following the notation of Proposition \ref{finite type}, the maximal ideal of $\E M$ is either
\begin{equation}\tag{0}
J(\E M)= (I_t\times \End_R(M_1))\cap \E M \mbox{ for some } t=1,\dots,n,
\end{equation}
or
\begin{equation}\tag{1}
J(\E M)=(\End_R(M_0)\times K_q)\cap \E M\mbox{ for some } q=1,\dots,m.
\end{equation}

Assume that $(0)$ holds and let $u=(u_0,u_1)$ be an element of $\E M$ such that $u_0$ is an automorphism of $M_0$. Then $u \notin J(\E M)$, because $u_0 \notin I_t$ for every $t=1,\dots,n$ (notice that $\bigcup_{t=1}^n I_t$ is the set of all non-invertible elements of $\End(M_0)$). In particular, $u_1$ is not in $\bigcup_{q=1}^m K_q$, that is, $u_1$ is an automorphism of $M_1$. This implies that $u$ is invertible in $\E M$.
In a similar way we can prove that if $(1)$ holds, then, for every $u=(u_0,u_1) \in \E M$, $u_1 \in \Aut(M_1)$ implies $u$ invertible in $\E M$.

Conversely, we want to prove that for every $u=(u_0,u_1)$, either $u$ or $1-u$ is invertible in $\E M$. Assume that there exists $i=0,1$ such that both conditions (a) and (b) hold. By (a), either $u_i$ or $1-u_i$ is invertible, so, by (b), either $u$ or $1-u$ is invertible in $\E M$.
\end{proof}

\section{Morphisms between two modules with local endomorphism rings}\label{vil}

Let $R$ be an arbitrary ring. We now consider the full subcategory $\Cal L$ of $\Mod R$ whose objects are all right $R$-modules with a local endomorphism ring. Let $$\Morph(\Cal L)$$ be the full category  of $\Morph(\Mod R)$ whose objects are all morphisms between two objects of $\Cal L$. The functor $U\colon \Morph(\Mod R)\to \Mod R\times \Mod R$ restricts to a functor $U\colon \Morph(\Cal L)\to \Cal L\times \Cal L$. Hence, for every object $M$ of $\Morph(\Cal L)$, the endomorphism ring of $M$ in the category $\Morph(\Cal L)$ is of type $\le 2$, and has at most two maximal ideals: the completely prime two-sided ideals
$$
I_{M,d}:=\{(u_0,u_1) \in \E M \mid u_0 \mbox{ is not an automorphism of }M_0\},
$$
and
$$
I_{M,c}:=\{(u_0,u_1) \in \E M \mid u_1 \mbox{ is not an automorphism of }M_1\}.
$$

As a consequence, an object $M$ of $\Morph(\Cal L)$ has  a local endomorphism ring if and only if either $I_{M,d}\subseteq I_{M,c}$ or $I_{M,d}\supseteq I_{M,c}$. Therefore, we get the following result.

\begin{Lemma}\label{6.1}
An object $M$ of $\Morph(\Cal L)$ has  a local endomorphism ring if and only if one of the following two conditions holds:

{\rm (1)} For every morphism $(u_0,u_1) \in \E M$, if $u_0$ is an automorphism of $M_0$, then $u_1$ is an automorphism of $M_1$, or

{\rm (2)} For every morphism $(u_0,u_1) \in \E M$, if $u_1$ is an automorphism of $M_1$, then $u_0$ is an automorphism of $M_0$.\end{Lemma}

The following two examples show that conditions $(1)$ and $(2)$ in Lemma~\ref{6.1} are independent, or, equivalently, that both proper inclusions $I_{M,d}\subset I_{M,c}$ and $ I_{M,c}\subset I_{M,d}$ can occur.

\begin{example}\label{exlocal}{\rm 
Let $\Z_p$ be the localization of $\Z$ at its maximal ideal $(p)$, so that $\Z_p$ is a discrete valuation domain, whose field of fractions is $\Q$. Consider the inclusion $\mu_M\colon \Z_p \hookrightarrow \Q$, viewed as a $\Z_p$-module morphism. Of course, $\End_{\Z_p}(\Z_p)\cong\Z_p$ and $\End_{\Z_p}(\Q)\cong\Q$, which are both local rings. It is immediate to see that the endomorphism ring of $M$ in $\M {\Z_p}$ is $E_M\cong \Z_p$, and that $0=I_{M,c}\subset I_{M,d}=p\Z_p$. }
\end{example}

\begin{example}\label{exlocal'}{\rm 
Let $\Z(p^\infty)$ be the Pr\"ufer group and $\mu_M\colon \Q\to \Z(p^\infty)$ be any group epimorphism, so that $\mu_M$ is an object $M$ in $\M \Z$. It is easily seen that the endomorphism ring $E_M$ of $M$ is canonically isomorphic to the localization $\Z_p$ of $\Z$ at its maximal ideal $(p)$. In this case, we have that $0=I_{M,d}\subset I_{M,c}=p\Z_p$.}
\end{example}

We will say that two objects $M$ and $N$ of $\M R$ belong to 

\noindent (1) {\em the same domain class}, and write $[M]_d=[N]_d$, if there exist morphisms $u\colon M\to N$ and $u'\colon N\to M$ such that $u_0\colon M_0\to N_0$ and $u'_0\colon N_0\to M_0$ are isomorphisms;

\noindent (2) {\em the same codomain class}, and write $[M]_c=[N]_c$, if there exist morphisms $u\colon M\to N$ and $u'\colon N\to M$ such that $u_1\colon M_1\to N_1$ and $u'_1\colon N_1\to M_1$ are isomorphisms.

Recall that a {\em completely prime ideal} $\Cal P$ of an additive category $\Cal C$ consists of a subgroup $\Cal P(A,B)$ of $\Hom_{\Cal C}(A,B)$, for every pair of objects of $\Cal C$, such that: $(1)$ for any objects $A,B,C$ of $\Cal C$, for every $f:A \rightarrow B$ and for every $g:B\rightarrow C$, one has that $gf \in \Cal P(A,C)$ if and only if either $f \in \Cal P(A,B)$ or $g\in \Cal P(B,C)$, and $(2)$ $\Cal P(A,A)$ is a proper subgroup of $\Hom_{\Cal C}(A,A)$ for every object $A$ of $\Cal C$.
If $A, B$ are objects of $\Cal C$, we say that $A$ and $B$ have {\em the same $\Cal P$ class}, and write $[A]_{\Cal P} = [B]_{\Cal P}$, if there exist right $R$-module morphisms $f : A \to B$ and $g: B\to A$ with $f \notin \Cal P(A,B)$ and $g\notin \Cal P(B, A)$  \cite[p.~565]{AlbPav5}.

In $\M {\Cal L}$ we have two completely prime ideals defined, for every pair of objects $\mu_M\colon M_0\rightarrow M_1$ and $\mu_N:N_0\rightarrow N_1$, by
$$
\Cal P_0(M,N):=\{u=(u_0,u_1): M\rightarrow N \mid u_0\mbox{ is not an isomorphism}\}
$$
and
$$
\Cal P_1(M,N):=\{u=(u_0,u_1): M\rightarrow N \mid u_1\mbox{ is not an isomorphism}\}.
$$
It is immediate to see that $M$ and $N$ have the same domain (resp.~codomain) class if and only if they have the same $\Cal P_0$ (resp.~$\Cal P_1$) class. Moreover, for every object $\mu_M:M_0\rightarrow M_1$ of $\M {\Cal L}$, $u \in \E M$ is an automorphism if and only if $u \notin \Cal P_0(M,M)\cup \Cal P_1(M,M)$. Then \cite[Theorem~6.2]{AlbPav5} implies the result that follows.

\begin{theorem}\label{6.1'}
Let $\mu_{M_k}\colon M_{0,k}\rightarrow M_{1,k}$, for $k=1,\dots,r$, and $\mu_{N_\ell} \colon N_{0,\ell}\rightarrow N_{1,\ell}$, for $\ell=1,\dots,s$, be $r+s$ objects in the category $\M {\Cal L}$. Then $\bigoplus_{k=1}^r M_k\cong \bigoplus_{\ell=1}^s N_\ell$ in the category $\M R$ if and only if $r=s$ and there exist two permutations $\varphi_d,\varphi_c$ of $\{1,2,\dots,r\}$ such that $[M_k]_d=[N_{\varphi_d(k)}]_d$ and $[M_k]_c=[N_{\varphi_c(k)}]_c$ for every $k=1,\dots, r$.
\end{theorem}

Let $n\ge 2$ be an integer. We will now give an example of a semilocal ring $R$ (of type $2n$) with $2n$ pairwise non-isomorphic right $R$-modules
$A_i,\ B_i$ (for $i=1,2,\dots,n$), all of them uniserial with local
endomorphism rings, and $n^2$ right $R$-module morphisms $\mu_{i,j}\colon
A_i\to B_j$ (for $i,j=1,2,\dots,n$), that is, objects $M_{i,j}$ of $\Morph(\Mod R)$  (for $i,j=1,2,\dots,n$), such that $\oplus_{i=1}^n M_{i,i}$ has
$n!$ pairwise non-isomorphic decompositions as a direct sum of $n$ indecomposable
objects of $\Morph(\Mod R)$. More precisely, we will see that the objects $M_{i,j}$ (for $i,j=1,2,\dots,n$) 
are such that:
\begin{enumerate}\item[{\rm (a)}] for every $i,j,k,\ell=1,2,\dots,n$,
$[M_{i,j}]_d=[M_{k,\ell}]_d$ if and only if $i=k$;
\item[{\rm (b)}] for every $i,j,k,\ell=1,2,\dots,n$,
$[M_{i,j}]_c=[M_{k,\ell}]_c$ if and only if $j=\ell$.\end{enumerate}
Therefore $$M_{1,1}\oplus M_{2,2}\oplus \dots\oplus M_{n,n}\cong
M_{\sigma(1),\tau(1)}\oplus M_{\sigma(2),\tau(2)}\oplus \dots\oplus M_{\sigma(n),
\tau(n)}$$ for every pair of permutations $\sigma,\tau$ 
of $\{1,2,\dots,n\}$. Our example is similar to \cite[Example~2.1]{TAMS}.

\begin{example}{\rm 
Let $p,q\in\Z$ be two distinct primes, $\Z_p,\Z_q$ be the localizations of $\Z$ at its maximal ideals $(p)$ and $(q)$, respectively, so that $\Z_p$ and $\Z_q$ are discrete valuation domains contained in $\Q$, and let $\Z_{pq}:=\Z_p\cap \Z_q$ be the subring of $\Q$ consisting of all rational numbers $a/b$, with $a,b\in\Z$ such that $p\nmid b$ and $q\nmid  b$. Thus $\Z_{pq}$ is a subring of $\Q$ that contains $\Z$, is a principal ideal domain,  is the localization of $\Z$ at the multiplicatively closed subset $\Z\setminus(p\Z\cup q\Z)$, is a semilocal ring of type $2$, and all its non-zero ideals are of the form $p^iq^j\Z_{pq}$, with $i,j\ge 0$.

Let $R$ denote the subring of ${\bf M}_n({\bf Q})$ whose elements are $n\times n$-matrices with
entries in ${{\bf Z}_{pq}}$ on and above the diagonal
and entries in ${{pq}{\bf Z}_{pq}}$ 
under the diagonal, that is, $$R=\cmat{cccc} {{\bf Z}_{pq}}&{{\bf Z}_{pq}}&\dots&{{\bf Z}_{pq}}\\
{pq}{{\bf Z}_{pq}}&{{\bf Z}_{pq}}&\dots&{{\bf Z}_{pq}}\\
\vdots& &\ddots&\\
{pq}{{\bf Z}_{pq}}&{pq}{{\bf Z}_{pq}}&\dots&{{\bf Z}_{pq}}
\fmat\subseteq {\bf M}_n({\bf Q}).$$ 

The set $W:=M_{1\times n}(\Q)$ of all $1\times n$ matrices with entries in $\Q$ is a right $R$-module under matrix multiplication. Set $$V_i:=
(\underbrace{q{\bf Z}_{{q}},\dots,q{\bf Z}_{{q}}}_{i-1},
\underbrace{{\bf Z}_{{q}},\dots,{\bf Z}_{{q}}}_{n-i+1}), \quad \mbox{for }\ i=1,2,\dots,n,$$ and $$X_j=(
\underbrace{p{\bf Z}_{{p}},\dots,p{\bf Z}_{p}}_{j-1},
\underbrace{{\bf Z}_{{p}},\dots,{\bf Z}_{{p}}}_{n-j+1}),\quad \mbox{for }\ j=1,2,\dots,m,$$ so that $V_i$ and $X_j$ are $R$-submodules of $W$ and $V_1\supset V_2\supset \dots\supset V_n\supset qV_1\supset
X_1\supset X_2\supset \dots\supset X_n\supset pX_1$. For every $i,j=1,2,\dots,n$, let $\mu_{i,j}\colon V_i\to W/X_j$ be the composite mapping of the inclusion $V_i\to W$ and the canonical projection $W\to W/X_j$, so that $\mu_{i,j}$ can be viewed as an object $M_{i,j}$ of $\Morph(\Mod R)$.

The endomorphism ring of the right $R$-module $V_i$ is isomorphic to the local ring $\Z_q$, because $V_i\cong e_{ii}R_{(q)}$ as an $R$-module, where $R_{(q)}$ denotes the localization of the $\Z_{pq}$-algebra $R$ at the maximal ideal $(q)$ of $\Z_{pq}$, so that $$\End_R(V_i)=\End_{R_{(q)}}(V_i)=\End_{R_{(q)}}( e_{ii}R_{(q)})\cong e_{ii}R_{(q)} e_{ii},$$ which is isomorphic to the localization $\Z_q$ of $\Z$ at its maximal ideal $q\Z$.

Let us prove that the endomorphism ring of the right $R$-module $W/X_j$ is also local. The module $W/X_j$ is isomorphic to $\Z(p^\infty)^n$ (direct sum of $n$ copies of the Pr\"ufer group $\Z(p^\infty))$ as an abelian group, hence is artinian as an abelian group, hence, it is artinian as a right $R$-module. For an artinian right $R$-module $L_R$, the restriction to the socle $\soc(L_R)$ is a local homomorphism $\End(L_R)\to\End(\soc(L_R))$, because every endomorphism of an artinian module $L_R$ which restricted to the socle is an automorphism of the socle, is necessarily an automorphism of $L_R$. As $pq$ is in the Jacobson radical of $R$, $pq$ annihilates all simple right $R$-modules, so that $\soc(W/X_j)$ is contained in $(\Z/p\Z)^n$. Now $(\Z/p\Z)^n$ is a uniserial right $R$-module of finite composition length $n$, whose socle is $(0,\dots,0, \Z/p\Z)$. Thus $\soc (W/X_j)=(0,\dots,0, \Z/p\Z)$, and the endomorphism ring of the socle of $W/X_j$ is isomorphic to the field $\Z/p\Z$ with $p$ elements. Thus there is a surjective local morphism $\End(W/X_j)\to\Z/p\Z$, hence $\End(W/X_j))$ is a local ring.

Let us show that, for every $i,j,k,\ell=1,2,\dots,n$,
$[M_{i,j}]_d=[M_{k,\ell}]_d$ if and only if $i=k$. The ring $R$ has type $2n$, so that it has $2n$ pairwise non-isomorphic simple right $R$-modules, up to isomorphism, $S_1,S_2,\dots, S_n$ (with $p$ elements each) and $T_1,T_2,\dots, T_n$  (with $q$ elements each).

The modules $V_i/qV_i$ are uniserial right $R$-modules of finite composition length $n$ and $q^n$ elements, their composition factors are the $n$ simple right $R$-modules $T_1,T_2,\dots, T_n$  (each with multiplicity one), and with top $V_i/\rad(V_i)$ isomorphic to $T_i$. Similarly, the modules $X_j/pX_j$ are uniserial right $R$-modules of finite composition length $n$ and $p^n$ elements, their composition factors are the $n$ simple right $R$-modules $S_1,S_2,\dots, S_n$  (each with multiplicity one), and with top $X_j/\rad(X_j)$ isomorphic to $S_j$.

It follows that the $2n$ right $R$-modules $V_1,\dots, V_n, W/X_1,\dots,W/X_n$ are pairwise non-isomorphic, that multiplication by $q$ is an isomorphism of $V_i$ onto $qV_i$, and  that multiplication by $p$ is an isomorphism of $W/X_j$ onto $W/pX_j$.

From the fact that the $2n$ right $R$-modules $V_1,\dots, V_n, W/X_1,\dots,W/X_n$ are pairwise non-isomorphic, it follows that, for every $i,j,k,\ell=1,2,\dots,n$,
$[M_{i,j}]_d=[M_{k,\ell}]_d$ implies $i=k$, and $[M_{i,j}]_c=[M_{k,\ell}]_c$ implies $j=\ell$.

Since multiplication by $q$ is an isomorphism of $V_i$ onto $qV_i$, we get, for every $ j\le\ell$, commutative squares $$\xymatrix{
V_i \ar[r]^{\mu_{ij}} \ar[dd]_p^{\cong} &   W/X_j\ar[d]_p^{\cong}  \\ & W/pX_j \ar[d]^{\rm can} \\
V_i\ar[r]_{\mu_{i\ell}}\  &   W/X_\ell
}\qquad\mbox{\rm and}\qquad \xymatrix{
V_i \ar[r]^{\mu_{i\ell}} \ar@{=}[d]&   W/X_\ell\ar[d]^{\rm can} \\
V_i\ar[r]_{\mu_{ij}}\  &   W/X_j
}.$$ This shows that $[M_{i,j}]_d=[M_{i,\ell}]_d$ for every $i,j,\ell$.

The fact that multiplication by $p$ is an isomorphism of $W/X_j$ onto $W/pX_j$ implies that, for every $i\le k$, there are commutative diagrams $$\xymatrix{
V_i \ar[r]^{\mu_{ij}} \ar[d]_q &   W/X_j\ar[d]_q^{\cong}  \\ 
V_k\ar[r]_{\mu_{kj}}&   W/X_j
}\qquad\mbox{\rm and}\qquad \xymatrix{
V_k \ar[r]^{\mu_{kj}} \ar@{^{(}->}[d]&   W/X_j\ar@{=}[d] \\
V_i\ar[r]_{\mu_{ij}}\  &   W/X_j.
}$$ These diagrams show that $[M_{i,j}]_c=[M_{k,j}]_c$ for every $i,j,k$.
}\end{example}

\section{Morphisms between uniserial modules}\label{7}

In this section we want to focus our attention on morphisms between uniserial modules. Recall that a right $R$ module $M$ is {\em uniserial} if the lattice of its submodules is linearly ordered under inclusion.

\begin{proposition}\label{enduni}
Let $\mu_M\colon M_0\rightarrow M_1$ be an object of $\Morph(\Mod R)$ with $M_0$ and $M_1$ non-zero uniserial right $R$-modules. Then $\E M$ has at most four maximal right (left) ideals, which are among the completely prime two-sided ideals
$$
I_0:=\{(u_0,u_1) \in \E M \mid u_0 \mbox{ is not an injective right }R\mbox{-module morphism}\},
$$
$$
I_1:=\{(u_0,u_1) \in \E M \mid u_1 \mbox{ is not an injective right }R\mbox{-module morphism}\},
$$
$$
K_0:=\{(u_0,u_1) \in \E M \mid u_0 \mbox{ is not a surjective right }R\mbox{-module morphism}\},
$$
and
$$
K_1:=\{(u_0,u_1) \in \E M \mid u_1 \mbox{ is not a surjectve right }R\mbox{-module morphism}\}.
$$
\end{proposition}

\begin{proof}
It immediately follows from \cite[Theorem~1.2]{TAMS} and Proposition \ref{finite type}.
\end{proof}

We can define four equivalences on $\Ob(\Morph(\Mod R))$ in the spirit of \cite{albfed}. For every pair of morphisms $\mu_M:M_0\rightarrow M_1$ and $\mu_{N}:N_0\rightarrow N_1$, we will write:

\noindent (1) 
$[M]_{0,m}=[N]_{0,m}$ if there exist $(u_0,u_1) \in \Hom(M,N)$ and $(v_0,v_1) \in \Hom(N,M)$ such that both $u_0$ and $v_0$ are injective right $R$-modules morphisms;

\noindent (2) 
$[M]_{1,m}=[N]_{1,m}$ if there exist $(u_0,u_1) \in \Hom(M,N)$ and $(v_0,v_1) \in \Hom(N,M)$ such that both $u_1$ and $v_1$ are injective right $R$-modules morphisms;

\noindent (3) 
$[M]_{0,e}=[N]_{0,e}$ if there exist $(u_0,u_1) \in \Hom(M,N)$ and $(v_0,v_1) \in \Hom(N,M)$ such that both $u_0$ and $v_0$ are surjective right $R$-modules morphisms;

\noindent (4) 
$[M]_{1,e}=[N]_{1,e}$ if there exist $(u_0,u_1) \in \Hom(M,N)$ and $(v_0,v_1) \in \Hom(N,M)$ such that both $u_1$ and $v_1$ are surjective right $R$-modules morphisms.

For morphisms between uniserial modules, we have the following weak form of the Krull-Schmidt Theorem. The proof is very similar to that of \cite[Proposition~4.1]{albfed} and is rather long, so we omit it.

\begin{theorem}\label{7.3}
Let $\mu_{M_j}\colon M_{0,j}\rightarrow M_{1,j}$, for $j=1,\dots,n$, and $\mu_{N_k}\colon N_{0,k}\rightarrow N_{1,k}$, for $k=1,\dots,t$, be $n+t$ morphisms between non-zero uniserial right $R$-modules. Then $\bigoplus_{j=1}^nM_j\cong \bigoplus_{k=1}^t N_k$ in $\Morph(\Mod R)$ if and only if $n=t$ and there exist four permutations $\varphi_{0,m},\varphi_{1,m},\varphi_{0,e}, \varphi_{1,e}$ of $\{1,2,\dots,n\}$ such that $[M_j]_{i,a}=[N_{\varphi_{i,a}(j)}]_{i,a}$ for every $j=1,\dots, n$, $i=0,1$ and $a=m,e$.
\end{theorem}

\end{document}